\documentclass[11pt]{amsart}%
\usepackage{graphicx}
\usepackage{amscd}
\usepackage{amsmath}
\usepackage{amsfonts}
\usepackage{xcolor}
\usepackage{amssymb}%
\providecommand{\U}[1]{\protect\rule{.1in}{.1in}}
\newtheorem{theorem}{Theorem}
\theoremstyle{plain}

\newtheorem{definition}{Definition}[section]

\newtheorem{remark}{Remark}[section]

\numberwithin{equation}{section}
\numberwithin{theorem}{section}

\begin{document}
\title[A New Characterization of Integral Powers of the Friedrichs-Legendre Operator]{A New Characterization of the Domains of Integral Powers of the Self-Adjoint Friedrichs-Legendre Operator}
\author{Lance L. Littlejohn}
\address{Department of Mathematics, Baylor University, Waco, TX 76706.}
\email{lance\_littlejohn@baylor.edu}
\author{Richard Wellman}
\address{Lucidy Sciences, 7984 S 1300 E, Sandy, UT 84094}
\email{wellman.richard@gmail.com}
\author{Quinn Wicks}
\address{Northern Virginia Community College, Manassas Campus MTS 148, 10950 Campus Dr,
Manassas, VA 20109}
\email{ewicks@nvcc.edu}
\dedicatory{ }
\date{July 17, 2026}
\begin{abstract}
Let $A$ be the self-adjoint operator in $L^{2}(-1,1)$, generated by the
second-order classical Legendre differential equation%
\[
\ell\lbrack y](t)=-\left(  (1-t^{2})y^{\prime}(t)\right)  ^{\prime}+ky(t)=\lambda
y(t)\quad(t\in(-1,1)),
\]
which has the Legendre polynomials $\{P_{m}\}_{m=0}^{\infty}$ as a complete sequence of eigenfunctions; here $k$ is a fixed, non-negative real number. This is the Friedrichs extension of the minimal operator associated with $\ell[\cdot]$ in $L^2(-1,1)$. For each $n \in \mathbb{N}$, we show that $\mathcal{D}(A^{n})$ is characterized by \textit{one} integrability condition instead of $2n$ boundary conditions as dictated by the classical Glazman-Krein-Naimark theory. We also prove that if $f\in\mathcal{D}(A^{n})$ then $f^{(n)}\in L^{2}(-1,1).$ This smoothness result extends known results when $n=1$ and $n=2.$ Furthermore, this result is optimal in the sense that there exists $g\in\mathcal{D}(A^{n})$ with $g^{(n+1)}\notin L^{2}(-1,1)$.  
\end{abstract}
\maketitle
\tableofcontents

\section{Introduction\label{Introduction}}

A natural setting for the study of the classical second-order Legendre expression
\begin{equation}
\begin{split}
\ell\lbrack y](t)  &  =-(1-t^{2})y^{\prime\prime}(t)+2ty^{\prime
}(t)+ky(t)\quad(t\in(-1,1))\\
&  =-\left(  (1-t^{2})y^{\prime}(t)\right)  ^{\prime}+ky(t)\quad(k\geq0\text{
is fixed)}
\label{Legendre DE}
\end{split}
\end{equation}
is the space $L^{2}(-1,1)$ with the usual inner product $(\cdot,\cdot)_{L^2(-1,1)}$ defined by
\[
(f,g)_{L^2(-1,1)}=\int_{-1}^{1}f(t)\overline{g}(t)\:dt.
\] 
This space is called the right-definite setting for the study of $\ell[\cdot]$. Historically it was Titchmarsh in the 1940's (see
\cite{Titchmarsh1} and \cite{Titchmarsh}) who first studied analytic
properties of (\ref{Legendre DE}) in $L^{2}(-1,1)$. In particular, Titchmarsh showed that the Legendre polynomials $\{P_{m}\}_{m=0}^{\infty}$ are
eigenfunctions of a self-adjoint operator $A,$ generated by $\ell\lbrack
\cdot],$ in $L^{2}(-1,1).$ Explicitly, this operator $A:\mathcal{D}(A)\subset
L^{2}(-1,1)\rightarrow L^{2}(-1,1)$ is given by%
\begin{align}
\begin{split}
(Af)(t)  &  =\ell\lbrack f](t)\quad(\text{a.e. }t\in(-1,1))\\
f\in\mathcal{D}(A)  &  =\{f:(-1,1)\rightarrow \mathbb{C}\mid f,f^{\prime }\in
AC_{\mathrm{loc}}(-1,1),f,\ell \lbrack f]\in L^{2}(-1,1);\\
&\hspace{3.5cm} \lim_{t\rightarrow \pm 1}(1-t^{2})f^{\prime }(t)=0\}.\end{split}
\label{Operator A}
\end{align}
In \cite{Everitt1980}, Everitt  
discussed the Legendre expression from both a right-definite and a (first)
left-definite point of view. The setting for a (first) left-definite study of
(\ref{Legendre DE}) is a certain Sobolev space $H_{1}$ with inner product (when $k>0$)%
\[
(f,g)_{1}:=\int_{-1}^{1}\left(  (1-t^{2})f^{\prime}(t)\overline{g}^{\prime
}(t)+kf(t)\overline{g}(t)\right)  dt.
\]
Other detailed analytical studies of the Legendre expression can be found in \cite{Everitt-Littlejohn-Wellman-Legendre2002}, \cite{Littlejohn-Zettl2011}, and in the Ph.D. theses of
Loveland \cite{Loveland} and Wicks \cite{Wicks}.

There has been a resurgence of interest in the past fifteen years in the analysis of
differential equations having orthogonal polynomial
eigenfunctions, mainly due to the emergence of exceptional orthogonal polynomials. This subject was introduced
in a seminal paper \cite{GUKM2009} in 2009 by D. G\'{o}mez-Ullate, N. Kamran, and R. Milson. Generally speaking, exceptional orthogonal polynomials form an
algebraically incomplete but analytically complete orthogonal sequence in
certain Hilbert spaces. Moreover, they are eigen-solutions of certain spectral-type
second-order differential equations with rational coefficients. Further information connecting the subjects of exceptional orthogonal polynomials and spectral theory of second-order differential operators may be found in \cite{GFGUM2019}, \cite{GUGM2021}, and \cite{LLMS2016}. 

In \cite{ELM-Legendre}, Everitt, Littlejohn, and Mari\'{c} prove that if $f \in \mathcal{D}(A)$, where $A$ and $\mathcal{D}(A)$ are given in (\ref{Operator A}), then $f^{\prime}\in L^{2}(-1,1)$. Hence, by appropriately redefining $f$ at the endpoints $t=\pm1$, if necessary, we see that $f \in AC[-1,1]$. Moreover, they show that there exists functions $f \in \mathcal{D}(A)$ such that $f'' \notin L^2(-1,1)$. This smoothness result was unexpected. Indeed, Weyl's limit-point/limit-circle theory governs how many solutions of a singular ordinary differential operator are $L^2$ integrable near the singularities but there are no general results available to determine the degree of differentiability of functions in the domain of such an operator. It is true that functions $f$ in the domain necessarily must be locally absolutely continuous (so $f'$ necessarily exists almost everywhere) but, because of possible blowup of $f$ at a singular endpoint, smoothness of functions on the entire interval should not be expected.

In \cite{Littlejohn-Wicks} and \cite{Wicks}, it is shown that if $f\in\mathcal{D}(A^{2}),$ then $f^{\prime\prime}\in L^{2}(-1,1)$ and, thus, $f,f^{\prime}\in AC[-1,1]$. However, as $n$ increases, the standard method for extending these results to $\mathcal{D}(A^n)$ would involve an indepth analysis of the sesquilinear form $[\cdot,\cdot]_{2n}$ (see (\ref{symplectic2n})) obtained from Green's formula for the $n^{th}$ integral power $\ell
^{n}[\cdot]$ of the Legendre expression (see (\ref{nth_power})-(\ref{Legendre-Stirling numbers})) and $2n$ appropriately chosen boundary conditions. As one might expect, as $n$ increases,
the sesquilinear form $[\cdot,\cdot]_{2n}$ becomes quite unwieldy.

The general left-definite theory developed by Littlejohn and Wellman (see \cite{FGHL2025}, \cite{Littlejohn-Wellman-LD-2002}, and \cite{Littlejohn-WellmanSpectra}) circumvents this complexity problem. In
\cite{Everitt-Littlejohn-Wellman-Legendre2002}, Everitt, Littlejohn, and Wellman, using results from this left-definite theory, characterize the domain of $\mathcal{D}(A^{n})$ in terms of \textit{one} integrability condition instead of, traditionally,
$2n$ appropriate  boundary conditions as dictated by the Glazman-Krein-Naimark (GKN) theory of self-adjoint extensions of symmetric ordinary differential operators. Specifically, they prove that 
\begin{align}
\mathcal{D}(A^n)=\{f:(-1,1)\rightarrow\mathbb{C}\mid f,f^{\prime},\ldots,&f^{(2n-1)}\in AC_{\textrm{loc}}(-1,1); \label{Domain(A^n)} \\ 
&(1-t^2)^{n}f^{(2n)}\in L^{2}(-1,1)\}.\notag
\end{align}
From this characterization of $\mathcal{D}(A^n)$, we prove the smoothness condition 
\begin{equation}
f\in\mathcal{D}(A^{n})\Longrightarrow f^{(n)}\in L^{2}(-1,1);
\label{D(A^n) Property}%
\end{equation}
(see Theorem \ref{Main Theorem}).
Of course, (\ref{D(A^n) Property}) implies that $f,f^{\prime}%
,\ldots,f^{(n-1)}\in AC[-1,1].$ This result generalizes the special cases $n=1$ and $n=2$. Moreover, as we will see, this result is also best possible. 

The contents of this paper are as follows. Section \ref{GKN theory} gives a brief account of the GKN theory of self-adjoint extensions of Lagrangian symmetrizable differential operators in a Hilbert space culminating in the statement of the GKN Theorem. In Section \ref{Right-Definite Theory - Legendre}, we review the classical right-definite theory in $L^{2}(-1,1)$ of the Legendre expression $\ell \lbrack\cdot]$ and the operator $A$ which has the Legendre polynomials $\{P_{m}\}_{m=0}^{\infty}$ as a complete set of eigenfunctions. In particular, we show that $A$ is bounded below in $L^{2}(-1,1)$ by $kI$. This feature of $A$ is important in order to apply the left-definite machinery developed in \cite{Littlejohn-Wellman-LD-2002, Littlejohn-WellmanSpectra} and later in \cite{FGHL2025}. The combinatorics of integral powers of the Legendre expression $\ell[\cdot]$ are discussed in Section \ref{Powers}; these results were also first established in \cite{Everitt-Littlejohn-Wellman-Legendre2002}. An interesting set of numbers called the Legendre-Stirling numbers appear as coefficients in the Lagrangian symmetric form of $\ell^n[\cdot]$; these numbers mimic properties of the classical Stirling numbers of the second kind. In Section \ref{A^n GKN}, we apply the GKN theory to obtain the classical characterization of the domain of $A^n$; this domain is determined by $2n$ appropriately chosen boundary conditions. Section \ref{Brief LD Theory} gives a brief survey of abstract left-definite theory developed in \cite{FGHL2025, Littlejohn-Wellman-LD-2002, Littlejohn-WellmanSpectra}. In Section \ref{Left-Definite Legendre Results}, we apply this left-definite theory to give the characterization (\ref{Domain(A^n)}) of $\mathcal{D}(A^{n})$. Full details of these results were first reported in \cite{Everitt-Littlejohn-Wellman-Legendre2002}. A key result in \cite{Everitt-Littlejohn-Wellman-Legendre2002}, as far as this manuscript is concerned, is the simple characterization of $\mathcal{D}(A^n)$ given in (\ref{Domain(A^n)}). Section \ref{Integral Inequality} discusses an important integral inequality (see Theorem \ref{Integral Inequality Theorem}) which is key to the two main results of this paper. In Section \ref{Main Results}, we present our two main results. Specifically, we give a new, simplified proof of (\ref{Domain(A^n)}) using Theorem \ref{Integral Inequality Theorem}. We then use this characterization to prove the implication in (\ref{D(A^n) Property}). Furthermore, we show that this smoothness condition is optimal in the sense that there exists $g\in\mathcal{D}(A^{n})$ with $g^{(n+1)}\notin L^{2}(-1,1).$ In Section \ref{Concluding Remarks}, we summarize our main results and emphasize how the general left-definite theory yields more efficient results than the classical methods presented by the GKN theory. 

Throughout this paper, $\mathbb{N}$ will denote the set of positive integers,
$\mathbb{N}_0=\mathbb{N\,\cup}\{0\},$ while $\mathbb{R}$ and $\mathbb{C}$
will denote, respectively, the real and complex number fields. For $m \in \mathbb{N}$, the set 
$AC_{\mathrm{loc}}^{(m)}(I)$ will denote those functions $f:I\rightarrow\mathbb{C}$
with $f,f',\cdots, f^{(m)}$ being absolutely continuous on all compact subintervals of an interval
$I\subset\mathbb{R}$; of course, in this case, the $(m+1)^{st}$ derivative $f^{(m+1)}$ exists almost everywhere (a.e.) on $I$. The condition $f\in AC[a,b]$ means that $f$ is
absolutely continuous on the compact interval $[a,b].$ If $w(t)>0$ (a.e. $t\in
I)$ is Lebesgue measurable, then $L^{2}(I;w)$ will denote the weighted Hilbert space of
all Lebesgue measurable functions $f:I\rightarrow\mathbb{C}$ satisfying
$\int_{I}\left\vert f(t)\right\vert ^{2}w(t)dt<\infty.$ Further notation is
introduced as needed throughout the paper.

\section{A Primer on GKN Theory}\label{GKN theory}

In this section, we give a short survey of the Glazman-Krein-Naimark (GKN) theory of self-adjoint extensions of Lagrangian symmetrizable ordinary differential equations. For a comprehensive study of this theory, we recommend the texts of Akhiezer and Glazman \cite[Chapter 8]%
{Akhiezer-Glazman}, Naimark \cite[Chapter V]{Naimark}, and Zettl \cite[Chapter 10]{Zettl-book}. For an overall comprehensive study of the spectral theory of Sturm-Liouville problems, the text of Gesztesy, Nichols, and Zinchenko \cite{GNZ} is an excellent source.

Suppose $m[\cdot]$ is the $2n^{th}$-order Lagrangian symmetrizable differentiable expression
\begin{equation}
m[y](t)=(1/w(t))\sum_{j=0}^n(-1)^{j}(a_{j}(t)y^{(j)}(t))^{(j)} \quad  (a.e.\thinspace t \in (a,b)),
\label{m expression}
\end{equation}
where, for simplicity (and our purposes), we assume each $a_{j} \in C^{j}(a,b)$ is real-valued and $w(t)$ is Lebesgue measurable and positive a.e. $t \in (a,b)$. In \cite{Naimark}, we note that more general conditions on the coefficients - namely quasi-differentiability - are assumed. 

The maximal domain for $m[\cdot]$ in $L^2((a,b);w)$ is defined to be
\begin{equation*}
\mathcal{D}(T_{\textrm{max}})=\left\{ f:(a,b) \rightarrow \mathbf{C}\vert f \in AC_{\textrm{loc}}^{(2n-1)}(a,b);
	f,m[f] \in L^2((a,b);w)\right\},
\end{equation*}
and the maximal operator $T_{\textrm{max}}: L^2((a,b);w) \rightarrow L^2((a,b);w)$ is given by
\[
T_{\textrm{max}}f  =m[f] \quad(f \in \mathcal{D}(T_{\textrm{max}})).
\]
The maximal operator is aptly named since $T_{\textrm{max}}$ maps the largest possible subspace of $L^{2}((a,b);w)$ into $L^{2}((a,b);w)$. Consequently, it is natural to call the adjoint $T_{\textrm{max}}^{\star}$ the minimal operator and it is denoted by $T_{\textrm{min}}$. Moreover $(T_{\textrm{min}})^{*}=T_{\textrm{max}}$.

The operator $T_{\textrm{min}}$ is a closed, symmetric operator in $L^2((a,b);w)$ and is specifically given by
\[
T_{\textrm{min}}f  =m[f] \quad (f \in \mathcal{D}(T_{\textrm{min}})),
\]
where
\[
\mathcal{D}(T_{\textrm{min}})=\{ f \in \mathcal{D}(T_{\textrm{max}}) \vert [f,g](b)-[f,g](a)=0\} \quad (g \in \mathcal{D}(T_{\textrm{max}})),
\]
and $[\cdot,\cdot]$ is the symplectic form defined by
\begin{align}
\begin{split}
[f,g](t) &:= \sum_{j=1}^{n}\sum_{m=1}^{j}(-1)^{m+j}\left( a_{j}(t)\overline
{g}^{(j)}(t)\right)^{(j-m)}f^{(m-1)}(t)\\
&- \sum_{j=1}^{n}\sum_{m=1}^{j}(-1)^{m+j}\left( a_{j}(t){f}^{(j)}%
(t)\right)^{(j-m)}\overline{g}^{(m-1)}(t).
\label{generalsymplectic}
\end{split}
\end{align}
For $f,g \in \mathcal{D}(T_{\textrm{max}})$, Green's formula yields the identity
\[
\int_{a}^{b}m[f](t)\overline{g}(t)w(t)dt-\int_{a}^{b}f(t)m[\overline g](t)w(t)dt =[f,g](b)-[f,g](a). 
\]
Notice, by definition of $\mathcal{D}(T_{\textrm{max}})$, the limits
\[
[f,g](b)=\lim_{t\rightarrow b}[f,g](t)\quad \text{and}\quad [f,g](a)=\lim_{t\rightarrow a}[f,g](t) 
\]
both exist and are finite for all $f,g \in \mathcal{D}(T_{\textrm{max}})$. We simplify notation by defining
\begin{equation}
[f,g]:=[f,g](b)-[f,g](a). \label{sesquilinear_difference}
\end{equation}
Since the coefficients of $m[\cdot]$ are real-valued, the deficiency indices $d_{\pm}(T_{\textrm{min}})$ of $T_{\textrm{min}}$, defined by
\[
d_{\pm}(T_{\textrm{min}}):=\textrm{dim}\{f \in \mathcal{D}(T_{\textrm{max}})\vert T_{\textrm{max}}f=\pm if\} \quad(i= \sqrt{-1}),
\]
are equal and satisfy the inequality $0 \le d_{\pm} \le 2n$. Consequently, from von Neumann's theory of self-adjoint extensions of symmetric operators (see \cite[Theorem 31, XII.4.31]{DunfordSchwartz}), self-adjoint extensions $S$ of $T_{\textrm{min}}$ exist and satisfy
\[
T_{\textrm{min}}\subset S=S^{\ast} \subset T_{\textrm{max}}. 
\]
We write $d=d_{\pm}(T_{\textrm{min}})$ and call $d$ the deficiency index of $T_{\textrm{min}}$. If $d=0$, then $T_{\textrm{min}}=T_{\textrm{max}}$ is self-adjoint but $d >0$ means that self-adjoint extensions of $T_{\textrm{min}}$ require $d$ boundary conditions (see Theorem \ref{GKN Theorem}) below. 

A set of functions $\{g_1,g_2,\cdots,g_{k}\} \subset \mathcal{D}(T_{\textrm{max}})$ is said to be linearly independent modulo $\mathcal{D}(T_{\textrm{min}})$ if
\[
\sum_{j=1}^{k}\alpha_{j}g_{j} \in \mathcal{D}(T_{\textrm{min}}) \Longrightarrow \text{each}\thinspace \alpha_{j}=0.
\]
In this case, we call $\{g_1,g_2,\cdots,g_k\}$ a set of Glazman boundary functions. If $k=d$, the deficiency index of $T_{\textrm{min}}$, we call this set a GKN set for $T_{\textrm{max}}$.

In terms of self-adjoint extensions of $T_{\textrm{min}}$ (equivalently, all self-adjoint restrictions of $T_{\textrm{max}}$) in $L^2((a,b);w)$, Glazman, Krein, and Naimark are credited with the following result, called the Glazman-Krein-Naimark theorem; see \cite[Theorem 4, Section 18.1]{Naimark}. 

\begin{theorem} $($\text{The GKN Theorem}$)$ \label{GKN Theorem} Suppose the minimal operator $T_{\textrm{min}}$ associated with the differential expression $m[\cdot]$, defined in (\ref{m expression}) has deficiency index $d$.
\begin{enumerate}
\item [(a)] Let $T$ be a self-adjoint extension of $T_{\textrm{min}}$ in $L^2((a,b);w)$ satisfying $T_{\textrm{min}}\subset T \subset T_{\textrm{max}}$. Then there exists a GKN set $\{g_1,g_2,\cdots, g_{d}\} \subset \mathcal{D}(T_{\textrm{max}})$ satisfying the Glazman symmetry conditions 
\begin{equation}
[g_{i},g_{j}]=0 \quad (i,j=1,2,\cdots,d) \label{GKN}
\end{equation}
(see (\ref{sesquilinear_difference})) such that $T$ has the explicit representation
\begin{align}
Tf  &=T_{\textrm{max}}[f]=m[f] \label{IV-4} \\
\mathcal{D}(T) &= \{f \in \mathcal{D}(T_{\textrm{max}}) \vert \thinspace [f,g_{j}]=0 \thinspace (j=1,\cdots,d))\}. \label{IV-5} 
\end{align}
\item [(b)] Let $\{g_1,g_2,\cdots g_{d}\} \subset \mathcal{D}(T_{\textrm{max}})$ be a GKN set for $T_{\textrm{max}}$ satisfying the conditions in (\ref{GKN}). Define the operator $T$ in $L^2((a,b);w)$ by (\ref{IV-4}) and (\ref{IV-5}). Then $T$ is a self-adjoint extension of $T_{\textrm{min}}$ in $L^2((a,b);w)$ and satisfies $T_{\textrm{min}} \subset T \subset T_{\textrm{max}}$.
\end{enumerate}
\end{theorem}

\section{Right-Definite Theory of the Legendre Differential
Expression\label{Right-Definite Theory - Legendre}}  

The Legendre polynomials $\{P_{m}\}_{m=0}^{\infty}$ form a complete orthonormal set in the Hilbert space $L^{2}(-1,1)$. In fact, with the $m^{th}$ Legendre polynomial defined by%
\begin{equation}
P_{m}(t)=\sqrt{\frac{2m+1}{2}}\sum_{j=0}^{[m/2]}\frac{(-1)^{j}(2m-2j)!}%
{2^{m}j!(m-j)!(m-2j)!}t^{m-2j}\quad(m\in\mathbb{N}_{0}), \label{Legendrepoly}%
\end{equation}
the sequence $\{P_{m}\}_{m=0}^{\infty}$ satisfies
\begin{equation}
(P_{m},P_{n})_{L^2(-1,1)}=\delta_{m,n}\quad(m,n\in\mathbb{N}_{0}), \label{Legendrenorm}%
\end{equation}
where $\delta_{m,n}$ is the Kronecker delta function. 

The $m^{th}$ Legendre polynomial $y=P_m(t)$ is a solution of 
\[
\ell[y](t)=(m(m+1)+k)y(t);
\]
that is, $\{P_m\}_{m=0}^\infty$ are `formal' eigenfunctions of the differential expression $\ell[\cdot]$ defined in (\ref{Legendre DE}). Various properties of the Legendre polynomials can be found in Szeg\"{o}'s classic text \cite[Chapter IV]{Szego}.

The maximal operator $T_{\textrm{max},2}:\mathcal{D}(T_{\textrm{max},2})\subset L^{2}(-1,1)\rightarrow
L^{2}(-1,1)$, generated by $\ell[\cdot]$, is
given by%
\begin{align*}
(T_{\textrm{max},2})(f)(t)  &  =\ell\lbrack f](t)\quad(\text{a.e }t\in(-1,1))\\
f\in\mathcal{D}(T_{\textrm{max},2})  &  =\{f:(-1,1)\rightarrow\mathbb{C}\mid f \in AC_{\textrm{loc}}^{(1)}(-1,1);f,\ell\lbrack f]\in L^{2}(-1,1)\}.\label{Max Operator1}
\end{align*}
Green's formula associated with $T_{\textrm{max},2}$ is%
\[
\int_{-1}^{1}\left(  (T_{\textrm{max},2})(f)(t)\overline{g}(t)-f(t)
{(T_{\textrm{max},2})[\overline{g}]}(t)\right)  dt=[f,g]_2(1)-[f,g]_2(-1)\quad(f,g\in\mathcal{D}%
(T_{\textrm{max},2})),
\]
where the symplectic form $[\cdot,\cdot]_2(\cdot)$ is the generalized Wronskian%
\begin{equation}
 [f,g]_2(t):=(1-t^{2})\left(  f(t)\overline{g}^{\prime}(t)-f^{\prime
}(t)\overline{g}(t)\right)  \quad(t\in(-1,1)).
\label{Legendre sesquilinear form}%
\end{equation}
By definition of $\mathcal{D}(T_{\textrm{max},2}),$ the limits%
\[
\lim_{t\rightarrow\pm1}[f,g]_2(t):=[f,g]_2(\pm1)\quad(f,g\in\mathcal{D}(T_{\textrm{max},2}))
\]
both exist and are finite.

The minimal operator $T_{\textrm{min},2}:\mathcal{D}(T_{\textrm{min},2})\subset L^{2}(-1,1)\rightarrow
L^{2}(-1,1)$ is defined to be%
\begin{align*}
(T_{\textrm{min},2}f)(t)  &  =\ell\lbrack f](t)\quad(\text{a.e }t\in(-1,1))\\
f\in\mathcal{D}(T_{\textrm{min},2})  &  =\{f\in\mathcal{D}(T_{\textrm{max},2})\mid [f,g]_2=0 \quad(f,g\in\mathcal{D}(T_{\textrm{max},2})\},
\end{align*}
where
\[
[f,g]_2 :=[f,g]_2(1)-[f,g]_2(-1).
\]
This Legendre differential expression $\ell\lbrack\cdot]$ is, using the
language of the Weyl classification (see \cite[Part 5]{Hellwig}), in the
limit-circle case at both endpoints $t=\pm1.$ Indeed, the principal and
nonprincipal solutions of
\[
\ell\lbrack y](t)=ky(t)
\]
are, respectively,%
\begin{equation}
u(t):=1\text{ and }v(t):=\dfrac{1}{2}\ln\bigg(\dfrac{1+t}{1-t}\bigg)
\quad(t\in(-1,1)). \label{LD-Motivation-PrincipalNonPrincipal}%
\end{equation}
\newline Both functions in $(\ref{LD-Motivation-PrincipalNonPrincipal}$) belong to $L^{2}(-1,1).$ By the Weyl
Alternative, the deficiency index of the minimal operator $T_{\textrm{min},2}$ is $2$.
This implies, from the general GKN theory, 
that two appropriate, singular boundary conditions involving the sesquilinear form
(\ref{Legendre sesquilinear form}) are required to produce any self-adjoint
extension in $L^{2}(-1,1)$ of the minimal operator $T_{\textrm{min},2}$. We note that there are uncountably many unbounded self-adjoint extensions $S$ in
$L^{2}(-1,1)$ of the minimal operator $T_{\textrm{min},2}$ satisfying%
\[
T_{\textrm{min},2}\subset S=S^{\ast}\subset T_{\textrm{min},2}^{\ast}=T_{\textrm{max},2}.
\]
The particular self-adjoint operator $A:\mathcal{D}(A)\subset L^{2}(-1,1)\rightarrow
L^{2}(-1,1)$ that we focus on in this paper is the Friedrichs extension defined in (\ref{Operator A}); specifically %
\begin{equation}
\begin{split}
(Af)(t)  &  =\ell\lbrack f](t)\quad(\text{a.e }t\in(-1,1))\\
f\in\mathcal{D}(A)  &  =\{f\in\mathcal{D}(T_{\textrm{max},2})\vert [f,g_1]_2 = [f,g_2]_2=0\},
\end{split}
\label{LegendreSA}
\end{equation}
where $\{g_1,g_2\} \subset C^2[-1,1]$ is the GKN set defined by
\[
g_{1}(t)=\left\{ 
\begin{array}{ll}
1 & t\text{ near }1 \\ 
0 & t\text{ near }-1%
\end{array}%
\right. \text{ and }g_{2}(t)=\left\{ 
\begin{array}{ll}
0 & t\text{ near }1 \\ 
1 & t\text{ near }-1.%
\end{array}%
\right. 
\]%
Notice that, for $f\in \mathcal{D}(A),$, we obtain separated boundary conditions for $A$, namely
\[
\lbrack f,g_{1}]_{2}=-\lim_{t\rightarrow 1}(1-t^{2})f^{\prime }(t)\text{ and 
}[f,g_{2}]_{2}=\lim_{t\rightarrow -1}(1-t^{2})f^{\prime }(t).\text{ }
\]
We refer the reader to another characterization of $\mathcal{D}(A)$ given in \cite[Section 6.2]{GLN} in which \textit{regularized} separated boundary conditions (or generalized boundary values) are given to describe the domain of $A$ using the principle and non-principle solutions defined in (\ref{LD-Motivation-PrincipalNonPrincipal}).

This operator $A$ has the Legendre polynomials $\{P_{m}\}_{m=0}^{\infty}$ as a
complete set of eigenfunctions and a discrete spectrum given explicitly by%
\[
\sigma(A)=\{m(m+1)+k\mid m\in\mathbb{N}_{0}\}.
\]

In \cite{ELM-Legendre}, the authors obtain the following characterization
theorem for $\mathcal{D}(A).$

\begin{theorem}
\label{ELM Theorem}Suppose $f\in\mathcal{D}(T_{\textrm{max},2})$. The
following statements are equivalent:
\end{theorem}

\begin{enumerate}
\item[(i)] $f \in\mathcal{D}(A)$.

\item[(ii)] $\lim_{t\rightarrow\pm1}(1-t^{2})f^{\prime}(t)=0.$

\item[(iii)] $f^{\prime}\in L^{2}(-1,1).$

\item[(iv)] $f^{\prime}\in L^{1}(-1,1).$

\item[(v)] $f$ is bounded on $(-1,1).$

\item[(vi)] $f\in AC[-1,1].$

\item[(vii)] $(1-t^{2})^{1/2}f^{\prime}\in L^{2}(-1,1).$
\end{enumerate}

Notice the striking equivalence of (iii) and (iv) in the above theorem. 
\begin{remark}
Observe from conditions (iii)-(vii) that the domain $\mathcal{D}(A)$ can be expressed in terms of just one integrability condition instead of two boundary conditions as demanded by the GKN theorem.
\end{remark}

For $f,g\in\mathcal{D}(A)$, one integration by parts, and an appeal to items
(ii) and (vi) of Theorem \ref{ELM Theorem} lead to the identity
\begin{align*}
(Af,g)_{L^2(-1,1)}  &  =\int_{-1}^{1}\left(  \left(  -(1-t^{2})f^{\prime}(t)\right)
^{\prime}+kf(t)\right)  \overline{g}(t)dt\\
&  =\left.  -(1-t^{2})f^{\prime}(t)\overline{g}(t)\right\vert _{-1}^{1}%
+\int_{-1}^{1}\left(  (1-t^{2})f^{\prime}(t)\overline{g}^{\prime
}(t)+kf(t)\overline{g}(t)\right)  dt\\
&  =\int_{-1}^{1}\left(  (1-t^{2})f^{\prime}(t)\overline{g}^{\prime
}(t)+kf(t)\overline{g}(t)\right)  dt;
\end{align*}
that is,
\begin{equation}
(Af,g)_{L^2(-1,1)}=\int_{-1}^{1}\left(  (1-t^{2})f^{\prime}(t)\overline{g}^{\prime
}(t)+kf(t)\overline{g}(t)\right)  dt\quad(f,g\in\mathcal{D}(A));
\label{Dirichlet's Identity}%
\end{equation}
this identity (\ref{Dirichlet's Identity}) is known as Dirichlet's identity. It follows, from (\ref{Dirichlet's Identity}), that
\begin{equation}
(Af,f)_{L^2(-1,1)}\geq k(f,f)_{L^2(-1,1)}\quad(f\in\mathcal{D}(A)); \label{Bounded below}%
\end{equation}
that is to say, $A$ is bounded below in $L^{2}(-1,1)$ by $kI,$ where $I$ is
the identity operator on $L^{2}(-1,1)$. This inequality is key to our left-definite analysis of the Legendre operator $A$. 
Sections \ref{Brief LD Theory} and \ref{Left-Definite Legendre Results} discuss this left-definite theory in detail.

\section{Powers of the Legendre Self-Adjoint Operator}\label{Powers}

Since our focus is in studying explicit properties of $\mathcal{D}(A^n)$, when $n \in \mathbb{N}$, we need to know the explicit form of $\ell^n[\cdot]$, the $n^{th}$ composite power of the Legendre expression $\ell[\cdot]$.

In \cite[Theorem 4.2]{Everitt-Littlejohn-Wellman-Legendre2002}, the authors show that, for $n\in\mathbb{N},$ the $n^{th}$ integral power $\ell^{n}[\cdot]$ of the
Legendre expression $\ell\lbrack\cdot]$ is given by%
\begin{equation}
\ell^{n}[y](t)=\sum_{j=0}^{n}(-1)^{j}\left(  c_{j}(n,k)(1-t^{2})^{j}%
y^{(j)}(t)\right)  ^{(j)}\label{nth_power},
\end{equation}
where%
\begin{equation}
c_{0}(n,k)=\left\{
\begin{array}
[c]{ll}%
0 & \text{if }k=0\\
k^{n} & \text{if }k>0,
\end{array}
\label{c_0}
\right. 
\end{equation}%
\begin{equation}
c_{j}(n,k)=\left\{
\begin{array}
[c]{ll}%
PS_{n}^{(j)} & \text{if }k=0\\
\sum_{r=0}^{j}\dbinom{n}{r}PS_{n-r}^{(j)}k^{r} & \text{if }k>0,
\end{array}
\label{c_j}
\right. 
\end{equation}
and where $PS_{n}^{(j)}$ is the Legendre-Stirling number, defined by
$PS_{0}^{(0)}=1$ and%
\begin{equation}
PS_{n}^{(j)}:=\sum_{r=1}^{n}(-1)^{r+j}\frac{(2r+1)(r^{2}+r)^{n}}%
{(r+j+1)!(j-r)!}\quad(n,j\in\mathbb{N},j\leq n). \label{Legendre-Stirling numbers}
\end{equation}
As shown in \cite[Theorem 4.1]{Everitt-Littlejohn-Wellman-Legendre2002}, the numbers $\{c_j(n,k)\}$ are the unique solutions to the recurrence relations
\begin{equation}
\sum_{j=0}^nc_j(n,k)\dfrac{(m+j)!}{(m-j)!}=(m(m+1)+k)^n; \label{RR}
\end{equation}
furthermore, when $k>0$, these numbers $\{c_j(n,k)\}$ are all positive.
The Legendre-Stirling numbers behave similarly to the classical Stirling
numbers of the second kind $S_{n}^{(j)}$(see \cite[Chapter V]{Comtet1974})$.$
For explicit information on the Legendre-Stirling numbers, we recommend the
contributions \cite{Andrews-Gawronski-Littlejohn2011},
\cite{Andrews-Littlejohn2009}, \cite{Egge}, \cite{Everitt-Littlejohn-Wellman-Legendre2002},
and\ \cite{Gawronski-Littlejohn-Neuschel} for further information on their
combinatorial properties and interpretations.

For $n \in \mathbb{N}$, the maximal operator $T_{\textrm{max},2n}$, generated by $\ell^n[\cdot]$ in $L^2(-1,1)$, is given by
\begin{equation*}
(T_{\textrm{max},{2n}}f)(t)=\ell^n[f](t)\quad(\text {a.e.}\thinspace t \in (-1,1))
\end{equation*}
with domain
\begin{equation*}
\mathcal{D}(T_{\textrm{max},{2n}})=\{f:(-1,1) \to \mathbb{C} \vert f \in AC_{\textrm{loc}}^{(2n-1)}(-1,1);f,\ell^n[f] \in L^2(-1,1)\}.
\end{equation*}
Green's formula associated with $T_{\textrm{max},2n}$ is%
\[
\int_{-1}^{1}\left(  T_{\textrm{max},2n})(f)(t)\overline{g}(t)-f(t)
{(T_{\textrm{max},2n})\overline{g}}(t)\right)  dt=[f,g]_{2n}(1)-[f,g]_{2n}(-1)\quad(f,g\in\mathcal{D}%
(T_{\textrm{max},2n})),
\]
where the symplectic form $[\cdot,\cdot]_{2n}(\cdot)$ (see also (\ref{generalsymplectic})) is given explicitly by%
\begin{align}
\begin{split}
[f,g]_{2n}(t) &= \sum_{j=1}^{n}\sum_{m=1}^{j}(-1)^{m+j}\left(  c_{j}(n,k)(1-t^{2})^{j}\overline
{g}^{(j)}(t)\right)^{(j-m)}f^{(m-1)}(t) \\ 
&- \sum_{j=1}^{n}\sum_{m=1}^{j}(-1)^{m+j}\left( c_{j}(n,k)(1-t^{2})^{j}f^{(j)}%
(t)\right)^{(j-m)}\overline{g}^{(m-1)}(t)
\label{symplectic2n}
\end{split}
\end{align}
for $f,g\in\mathcal{D}(T_{\max,2n})$ and $-1<t<1$.
When $n=1$, this symplectic form is given in (\ref{Legendre sesquilinear form}). 
We recommend \cite{Littlejohn-Wicks} and \cite{Wicks} for a comprehensive discussion of the spectral analysis of $A^2$.

Notice, by definition of $\mathcal{D}(T_{\textrm{max},2n})$, that the limits%
\[
\lim_{t\rightarrow\pm1}[f,g]_{2n}(t):=[f,g]_{2n}(\pm1)\quad(f,g\in\mathcal{D}(T_{\textrm{max},2n}))
\]
both exist and are finite. For simplicity, define
\begin{equation}
[f,g]_{2n}:= [f,g]_{2n}(1)-[f,g]_{2n}(-1). \label{Legendre Symplectic}
\end{equation}

The minimal operator $T_{\textrm{min},2n}:\mathcal{D}(T_{\textrm{min},2n})\subset L^{2}(-1,1)\rightarrow
L^{2}(-1,1)$ is defined to be%
\begin{align*}
(T_{\textrm{min},2n}f)(t)  &  =\ell^n\lbrack f](t)\quad(\text{a.e.}\thinspace t\in(-1,1))\\
f\in\mathcal{D}(T_{\textrm{min},2n})  &  =\{f\in\mathcal{D}(T_{\textrm{max},2n})\mid [f,g]_{2n}=0\quad(g \in \mathcal{D}(T_{\textrm{max},2n})\}.
\end{align*}
As discussed in Section \ref{GKN theory}, $T_{\textrm{min},2n}$ is a closed, symmetric operator in $L^{2}(-1,1)$
satisfying $T_{\textrm{min},2n}^{\ast}={T_{\textrm{max},2n}}$ and $T_{\textrm{max},2n}^{\ast}={T_{\textrm{min},2n}}$.

\section{A GKN Characterization of Integral Powers of the Legendre Operator $A$}\label{A^n GKN}

I. M. Glazman (\cite{Glazman1950} and \cite[pp. 24--27]{Glazman}) showed that if $S$ is a symmetric operator in a Hilbert space $H$ with finite deficiency indices $d_{\pm}(S)$, then $d_{\pm}(S^2)=d_{+}(S)+d_{-}(S)$; see also the recent contribution by \cite{FGL_Deficiency_Indices}. 

Recall from Section \ref{Right-Definite Theory - Legendre} that the minimal operator $T_{\textrm{min},2}$, generated by the Legendre differential expression $\ell[\cdot]$, has equal deficiency indices $d(A)=d_{+}(A)=d_{-}(A)=2.$ Repeated applications of Glazman's result shows that, for $n \in \mathbb{N}$, the deficiency indices of $A^{n}$ are equal and 
\begin{equation}
d(A^n)=d_{\pm}(A^n)=2n.
\end{equation}
Consequently, from the GKN theorem (Theorem \ref{GKN Theorem}), for each fixed $n\in 
\mathbb{N},$ there are $2n$ boundary conditions needed to characterize the domain
of $A^{n}$. More specifically, there exists GKN sets $\{g_{j}\}_{j=1}^{2n}\subset \mathcal{D}(T_{\mathrm{\max }%
\text{,}2n})$ satisfying

\begin{enumerate}
\item[(i)] $\{g_{j}\}_{j=1}^{2n}$ is linearly independent modulo $T_{%
\mathrm{\min }\text{,}2n}$ \text{and}

\item[(ii)] $[g_{i},g_{j}]_{2n}=0$ for $i,j=1,2,\ldots, 2n,$
\end{enumerate}

\noindent such that 
\begin{equation}
\mathcal{D}(A^{n})=\{f\in \mathcal{D}(T_{\mathrm{\max }\text{,}2n})\mid
\lbrack f,g_{i}]_{2n}=0\quad (i=1,2,\ldots 2n)\} \label{LegendreGKN},
\end{equation}%
where $[f,g]_{2n}$ is defined in (\ref{Legendre Symplectic}). We show, in Section \ref{Left-Definite Legendre Results}, that these $%
2n$ boundary conditions can be replaced by \textit{one} integrability
condition.

\begin{remark}\label{GKN sets for Legendre}
The usual GKN set associated with $\mathcal{D}(T_{\mathrm{max},2n})$ that describes $\mathcal{D}(A^n)$ is $\{g_j\}_{j=1}^{2n} \subset C^{2n}[-1,1]$, where 
\[
g_{j}(x)=\left\{ 
\begin{array}{ll}
P_{j-1}(x) & \text{near }x=-1 \\ 
0 & \text{near }x=1%
\end{array}%
\right. \quad (j=1,2,\ldots ,n),
\]%
\[
g_{n+j}(x)=\left\{ 
\begin{array}{ll}
0 & \text{near }x=-1 \\ 
P_{j-1}(x) & \text{near }x=1
\end{array}
\right. \quad (j=1,2,\ldots ,n),
\]
and where $\{P_{j-1}\}_{j=1}^{n}$ are Legendre polynomials.  Of course, by
linearity, each $P_{j-1}(x)$ above can be replaced by the monomial $x^{j-1}$. 
\end{remark}



\section{A Brief Introduction to Left-Definite Theory\label{Brief LD Theory}}
Left-definite theory has its roots in the theory of second-order Sturm--Liouville
differential equations. Indeed, consider the Sturm--Liouville (SL) generalized eigenvalue equation
\[
(\tau f )(t) := w(t)^{-1} ((-p(t)f'(t))' + q(t)) f(t) = \lambda f(t)  
\]
for a.e. $t \in (a,b)$, where $(a,b)$ is an interval, bounded or unbounded, of the real line $\mathbb{R}$. Here, we make the usual assumptions that $1/p, q, r$ are real-valued and locally integrable on $(a, b)$, with $p > 0$ a.e. on $(a, b)$ and $f \in AC_{\textrm{loc}}(a, b)$ with quasi-derivative $f^{[1]} := pf' \in AC_{\textrm{loc}}(a, b)$. The two
obvious sesquilinear forms in this case are
\begin{equation}
\int_{a}^{b}f(t)\overline{g}(t)w(t)dt\text{ and }\,\int_{a}^{b}\left(
p(t)f^{\prime}(t)\overline{g}^{\prime}(t)+q(t)f(t)\overline{g}(t)\right)  dt.
\label{two sesquilinear forms}%
\end{equation}
If $w>0$ a.e.~on $(a,b)$, this SL problem is called \textit{right-definite}
since in this case the sesquilinear form
\begin{equation*}
\int_{a}^{b}f(t)\overline{g}(t)w(t)dt\quad(f,g\in L^{2}((a,b);w))%
\end{equation*}
is positive in the sense that it defines the usual scalar product in the Hilbert
space $L^{2}((a,b);w\thinspace)$. The problem is called \textit{left-definite} if the
sesquilinear form
\begin{equation*}
\int_{a}^{b}\left(  p(t)f^{\prime}(t)\overline{g}{^\prime}(t)%
+q(t)f(t)\overline{g}(t)\right)dt %
\end{equation*}
is positive, that is, it defines a positive scalar product. When a SL problem
is both right-definite and left-definite, the two sesquilinear forms in $(\ref{two sesquilinear forms})$ are connected
through Dirichlet's formula
\[%
\begin{split}
\left(  \tau f,g\right)  _{L^{2}((a,b);w)}  &  =\int_{a}^{b}(\tau
f)(t)\overline{g}(t)w(t)dt\\
&  =\int_{a}^{b}\left(  -(p(t)f^{\prime}(t))^{\prime}+q(t)f(t)\right)
\overline{g}(t)dt\\
&  =\int_{a}^{b}\left(  p^{-1}(t)f^{[1]}(t)\overline{g}^{[1]}(t)%
+q(t)f(t)\overline{g}(t)\right)  dt\\
&  =\int_{a}^{b}\left(  p(t)f^{\prime}(t)\overline{g}^{\prime}(t)%
+q(t)f(t)\overline{g}(t)\right)  dt:=(f,g)_{1};
\end{split}
\]
here $f,g\in AC_{\mathrm{loc}}(a,b)$, $pf^{\prime},pg^{\prime}\in
AC_{\mathrm{loc}}(a,b)$ and $f,g$ have compact support in $(a,b)$. In the
literature, the inner product $(\cdot,\cdot)_{1}$, extended to elements in an
appropriate Hilbert-Sobolev function space, is called the
left-definite inner product. Littlejohn and Wellman
refer to $(\cdot.\cdot)_{1}$ as the \textit{first} left-definite inner product as we explain below.

In this section, we give a brief description of a general left-definite operator theory; we highlight only the main left-definite results needed for this
paper; for a full description of left-definite theory, see
\cite{Littlejohn-Wellman-LD-2002}, \cite{Littlejohn-WellmanSpectra}, and more recently, \cite{FGHL2025}.

Throughout this section, we assume that $H=(V,(\cdot,\cdot))$ is a Hilbert
space, where $V$ is a vector space over $\mathbb{C}.$ Let $A:\mathcal{D}%
(A)\subset H\rightarrow H$ be a self-adjoint operator, bounded or unbounded, that is bounded below by
$kI$ for some $k>0;$ that is,
\[
(Ax,x)\geq k(x,x)\quad(x\in\mathcal{D}(A)).
\]
It follows that $A^{r},$ for each $r>0,$ is a self-adjoint operator which is
bounded below in $H$ by $k^{r}I.$ Using the Hilbert space spectral theorem
(see \cite[Chapters 12 and 13]{Rudin}), the authors in
\cite{Littlejohn-Wellman-LD-2002} and \cite{Littlejohn-WellmanSpectra}
construct a continuum of Hilbert spaces $\{H_{r}\}_{r>0}$; we call $H_{r}$
the $r^{th}$ left-definite space associated with $(H,A)$. The more recent publication \cite{FGHL2025}, using a model theory approach to
left-definite theory, establishes the main results of left-definite theory
from another point of view and adds several new examples. 

Suppose, for $r>0,$ $V_{r}$ is a subspace of $V$ and let $(\cdot,\cdot)_{r}$
and $\left\Vert \cdot\right\Vert _{r}$ denote an inner product and associated
norm, respectively, over $V_{r}$ (quite possibly different from $(\cdot
,\cdot)$ and $\left\Vert \cdot\right\Vert $). We denote the resulting inner
product space by $H_{r}=(V_{r},\left(  \cdot,\cdot\right)  _{r}).$

We now define an $r^{th}$ left-definite space associated with $(H,A).$

\begin{definition}
\label{Definition of rth LD space} Let $r>0$ and suppose $V_{r}$ is a subspace
of the Hilbert space $H$ $=(H,(\cdot,\cdot))$ and $\left(  \cdot,\cdot\right)
_{r}$ is an inner product on $V_{r}.$ Let $H_{r}=(V_{r},(\cdot,\cdot)_{r}).$
We say that $H_{r}$ is an $r^{th}$ left-definite space\textbf{ }%
\textit{associated with the pair }$(H,A)$ if each of the following conditions
hold:\newline(1) $H_{r}$ is a Hilbert space,\newline(2) $\mathcal{D}(A^{r})$
is a subspace of $V_{r},$ \newline(3) $\mathcal{D}(A^{r})$ is dense in
$H_{r},$ \newline(4) $\left(  x,x\right)  _{r}\geq k^{r}\left(  x,x\right)
\quad(x\in V_{r}),$ and\newline(5) $\left(  x,y\right)  _{r}=\left(
A^{r}x,y\right)  \quad(x\in\mathcal{D}(A^{r}),\;y\in V_{r}).$
\end{definition}

The following theorem, proved in \cite[Theorem 3.1]{Littlejohn-Wellman-LD-2002}, clarifies the existence and uniqueness of left-definite spaces and establishes that the nest $\{H_r\}_{r>0}$ is a Hilbert scale. 

\begin{theorem}
\label{E/U LD Theorem} Suppose $A:\mathcal{D}(A)\subset H\rightarrow H$ is a self-adjoint
operator that is bounded below by $kI$ for some $k>0.$ Let $r>0.$ Define
$H_{r}=(V_{r},(\cdot,\cdot)_{r})$ by%
\begin{equation}
V_{r}=\mathcal{D}(A^{r/2}), \label{V_r}%
\end{equation}
and%
\[
(x,y)_{r}=(A^{r/2}x,A^{r/2}y)\quad(x,y\in V_{r}).
\]
Then $H_{r}$ is an $r^{\text{th}}$ left-definite space associated with the pair $(H,A).$
Moreover, if $H_{r}^{\prime}:=(V_{r}^{\prime},(\cdot,\cdot)_{r}^{\prime})$ is
another $r^{th}$ left-definite space associated with the pair $(H,A),$ then
$V_{r}=V_{r}^{\prime}$ and $(x,y)_{r}=(x,y)_{r}^{\prime}$ for all $x,y\in
V_{r}=V_{r}^{\prime};$ i.e. $H_{r}$ is uniquely determined by the five
conditions in Definition \ref{Definition of rth LD space}. \newline In particular,
\begin{equation}
\mathcal{D}(A^{1/2})=V_{1},\mathcal{D}(A)=V_{2},\text{ and, for $n \in \mathbb{N},$
}\mathcal{D}(A^{n})=V_{2n}. \label{V_(2n)}%
\end{equation}
Moreover,

\begin{enumerate}
\item[(1)] \textit{suppose }$A$\textit{ is bounded. Then, for each} $r>0,$

\begin{enumerate}
\item[(i)] $V=V_{r};$

\item[(ii)] \textit{the inner products} $(\cdot,\cdot)$ and $(\cdot,\cdot
)_{r}$ \textit{are equivalent};
\end{enumerate}

\item[(2)] \textit{suppose }$A$\textit{ is unbounded}. \textit{Then}

\begin{enumerate}
\item[(i)] $V_{r}$\textit{ is a proper subspace of }$V;$

\item[(ii)] $V_{s}$\textit{ is a proper subspace of }$V_{r}$\textit{ whenever
}$0<r<s;$

\item[(iii)] \textit{the inner products }$(\cdot,\cdot)$\textit{ and }%
$(\cdot,\cdot)_{s}$\textit{ are not equivalent for any }$s>0;$

\item[(iv)] \textit{the inner products }$(\cdot,\cdot)_{r}$\textit{ and
}$(\cdot,\cdot)_{s}$\textit{ are not equivalent for any }$r,s>0,$ $r\neq s.$
$\square$
\end{enumerate}
\end{enumerate}
\end{theorem}

\section{A Left-Definite Characterization of Integral Powers of the Legendre Operator $A$\label{Left-Definite Legendre Results}}

From (\ref{Bounded below}), the Legendre self-adjoint operator $A$, defined in
(\ref{LegendreSA}), is bounded below by $kI$ in $L^{2}(-1,1).$ Consequently, when $k>0$, 
the left-definite theory discussed in the last section applies to $A.$ In
\cite{Everitt-Littlejohn-Wellman-Legendre2002}, the authors explicitly determine the
left-definite spaces $\{H_{n}\}_{n\in\mathbb{N}}$ of
associated with $(L^{2}(-1,1),A)$ for all $n \in \mathbb{N}$. Even though Theorem \ref{E/U LD Theorem} asserts the existence of these left-definite Legendre spaces for all $r>0$, we can only determine these spaces for $n \in \mathbb{N}$ (on account of (5) in Definition \ref{Definition of rth LD space}). 

In \cite[Theorems 5.1, 5.2, 5.3, 5.4]{Everitt-Littlejohn-Wellman-Legendre2002}), the authors prove the following theorem. 

\begin{theorem}
\label{LD_Legendre spaces} For $n\in\mathbb{N},$ the $n^{th}$ left-definite
space $H_{n}=(V_{n},(\cdot,\cdot)_{n})$ associated with $(L^{2}(-1,1),A)$ is
given by%
\begin{eqnarray}
V_{n} &=&\{ f:(-1,1) \rightarrow \mathbb{C} \vert f \in AC_{\textrm{loc}}^{(n-1)}(-1,1);  \label{V_n_1} \\
&&\qquad \qquad \qquad \qquad (1-t^2)^{j/2}f^{(j)} \in L^2(-1,1)\thinspace (j=0,1,\cdots, n)\} \nonumber \\
&=& \left\{f:(-1,1) \rightarrow \mathbb{C} \vert f \in AC_{\textrm{loc}}^{(n-1)}(-1,1); (1-t^2)^{n/2}f^{(n)} \in L^2(-1,1)\right\}. \label{V_nalt}
\end{eqnarray}
The inner product $(\cdot,\cdot)_{n}$ is explicitly given by
\begin{equation}
(f,g)_{n}=\sum_{j=0}^{n}c_{j}(n,k)\int_{-1}^{1}f^{(j)}(t)\overline{g}%
^{(j)}(t)(1-t^{2})^{j}dt\quad(f,g\in V_{n}) \label{Legendre LD IP I},
\end{equation}
where the coefficients $c_{j}(n,k)$ are defined in (\ref{c_0}) and
(\ref{c_j}). 
Moreover, the Legendre polynomials $\{P_{m}\}_{m=0}^{\infty}$
form a complete orthogonal set in each $H_{n}$ and the Legendre polynomials $\{P_m\}_{m=0}^\infty$ satisfy the orthogonality relationship 
\begin{equation}
(P_m,P_r)_n = \sum_{j=0}^nc_j(n,k)\int_{-1}^{1}\dfrac{d^j(P_m(t)}{dt^j}\dfrac{d^j(P_r(t))}{dt^j}(1-t^2)^jdt=(m(m+1)+k)^j\delta_{m,r}.
\label{LD Orthog}
\end{equation}
In particular, we see from (\ref{V_(2n)}), (\ref{V_n_1}), and (\ref{V_nalt}) that the identity in (\ref{D(A^n) Property}) holds; that is, 
\begin{equation}
\mathcal{D}(A^n)=\{f:(-1,1) \rightarrow \mathbb{C} \vert f \in AC_{\textrm{loc}}^{(2n-1)}(-1,1); (1-x^n)f^{(2n)} \in L^2(-1,1)\}.
\label{DomainA^n}
\end{equation}
\end{theorem}
$\square$

In \cite[Section 7.4]{Triebel} (see also \cite{Tr70} and \cite{Tr70a}), Triebel determines $\mathcal{D}(A^r)$ for \underline{all} $r>0$. His remarkable results go beyond the work in \cite{Everitt-Littlejohn-Wellman-Legendre2002}. In particular, he establishes
\[
\mathcal{D}(A^n)=V_n \quad(n \in \mathbb{N}),
\]
where $V_n$ is defined in (\ref{V_n_1}). His interests in powers of $A^r$ were, however, domain-related only. When $r=n \in \mathbb{N}$, Triebel makes no mention of the form of $A^n$ which is given in (\ref{nth_power}). In fact, the form of $A^r$ when $r>0$ but $r \notin \mathbb{N}$ remains open.

We will give a new proof of (\ref{DomainA^n}) - one of our two main results in this paper given in Section \ref{Main Results}. The new proof is simpler than that given in \cite{Everitt-Littlejohn-Wellman-Legendre2002} and follows from applying a well-known integral inequality that we discuss in the next section. This integral inequality will also be used to prove our other main result. Indeed, in Section \ref{Main Results}, we show that functions in $\mathcal{D}(A^n)$ have considerable smoothness properties, generalizing earlier work of Everitt, Littlejohn, and Mari\'{c} (see \cite{ELM-Legendre}) for the second-order case and Littlejohn and Wicks (see \cite{Littlejohn-Wicks}) for the fourth-order case. More specifically, we show that when $f \in \mathcal{D}(A^n)$, then $f^{(n)} \in L^2(-1,1)$. Additionally, as we will see, this is an optimal result in the sense that there exists $g \in \mathcal{D}(A^n)$ with $g^{(n+1)} \notin \mathcal{D}(A^n)$.

\begin{remark}
\label{V_n comment} The equivalence of (\ref{V_n_1}) and (\ref{V_nalt}) is
surprising. The proof is embedded in our proof that the Legendre polynomials are complete in each $H_n$ (\cite[Theorem 5.2]{Everitt-Littlejohn-Wellman-Legendre2002}).
\end{remark}
\begin{remark}
The inner product in (\ref{Legendre LD IP I}) is generated from the inner product $(\ell^n[f],g)_{L^2(-1,1)}$ (as required by (5) in Definition (\ref{Definition of rth LD space})). In fact, for polynomials $p$ and $q$ it is straightforward to show, via integration by parts, that
\begin{equation}
(\ell^n[p],q)_{L^2(-1,1)}=\sum_{j=0}^nc_j(n,k)\int_{-1}^1 (1-t^2)^jp^{(j)}(t)\bar{q}^{(j)}(t)dt;
\end{equation}
that is to say,
\begin{equation}
(\ell^n[p],q)_{L^2(-1,1)}=(p,q)_n. \label{LD-poly}
\end{equation}
\end{remark}
The reader can see that the characterization of $\mathcal{D}(A^n)$ in (\ref{DomainA^n}) is quite different from the classical GKN characterization of $\mathcal{D}(A^n)$ given by (\ref{LegendreGKN}). Indeed, to check whether a function $f$ belongs to $\mathcal{D}(A^n)$, we need only check one integrability condition instead of $2n$ boundary conditions using the sesquilinear form (\ref{symplectic2n}). 

\section{An Integral Inequality\label{Integral Inequality}}
Key to our two main results in this paper is the following integral inequality which can be found in \cite{Chisholm-Everitt1971} (see also  \cite{Muckenhoupt1972}, \cite{Talenti1969}, and \cite{Tomaselli1969}). 

\begin{theorem}
\label{Integral Inequality Theorem} Let $(a,b)$ be an open interval of the
real line, bounded or unbounded, and suppose $w(t)>0$ (a.e. $t\in(a,b)$) is
Lebesgue measurable. In addition, suppose $\varphi,\psi:(a,b)\rightarrow
\mathbb{C}$ satisfy the conditions

\begin{enumerate}
\item[(i)] $\varphi,\psi:(a,b)\rightarrow\mathbb{C}\in L_{\mathrm{loc}}%
^{2}(a,b);$

\item[(ii)] there exists $c\in(a,b)$ such that $\varphi\in L^{2}((a,c];w)$ and
$\psi\in L^{2}([c,b);w);$

\item[(iii)] for all $[\alpha,\beta]\subset(a,b),$%
\[
\int_{\alpha}^{\beta}\left\vert \varphi(t)\right\vert ^{2}w(t)dt>0\text{ and
}\int_{\alpha}^{\beta}\left\vert \psi(t)\right\vert ^{2}w(t)dt>0.
\]

\end{enumerate}

\noindent Define the linear integral operators $S,T:L^{2}((a,b);w)\rightarrow
L_{\mathrm{loc}}^{2}((a,b);w)$ by%
\[
(Sf)(t)=\varphi(t)\int_{t}^{b}\psi(u)f(u)w(u)du\quad(t\in(a,b);f\in
L^{2}((a,b);w))
\]
and%
\[
(Tf)(t)=\psi(t)\int_{a}^{t}\varphi(u)f(u)w(u)du\quad(t\in(a,b);f\in
L^{2}((a,b);w)).
\]
Moreover, define $K:(a,b)\rightarrow(0,\infty)$ by%
\[
K(t)=\left(  \int_{a}^{t}\left\vert \varphi(u)\right\vert ^{2}w(u)du\right)
^{1/2}\left(  \int_{t}^{b}\left\vert \psi(u)\right\vert ^{2}w(u)du\right)
^{1/2}\quad(t\in(a,b))
\]
and let $K\in\lbrack0,\infty]$ be given by%
\[
K=\sup_{t\in(a,b)}K(t).
\]
Then a necessary and sufficient condition that $S$ and $T$ are bounded
operators from $L^{2}((a,b);w)$ into $L^{2}((a,b);w)$ is that%
\[
0<K<\infty.
\]
Equivalently, $K(t)$ is bounded on $(a,b).$ $\square$
\end{theorem}

\section{Main Results\label{Main Results}}

In this section, we give a new proof (Theorem \ref{Char. of V_n}) of $V_n = V_n'$, where $V_n$ and $V_n'$ are given, respectively, in (\ref{V_n_1}) and (\ref{V_nalt}) (and below in Theorem \ref{Char. of V_n}). The equality of these two sets yields the simple characterization of $\mathcal{D}(A^n)$ given in (\ref{D(A^n) Property}). In this section, we also establish optimal smoothness for functions $f \in \mathcal{D}(A^n)$; see Theorem \ref{Main Theorem} below.

\begin{theorem} \label{Char. of V_n} 
Let $n \in \mathbb{N}$. We recall the definition of $V_n$ in (\ref{V_n_1}):
\begin{align}
\begin{split}
V_n = \{f:(-1,1)\rightarrow \mathbb{C}\vert &f \in AC_{\textrm{loc}}^{(n-1)}(-1,1);\\
&(1-x^2)^{j/2}f^{(j)} \in L^2(-1,1) \quad (j=0,1,\cdots,n)\}.
\end{split}   
\label{V_n def}
\end{align}
Let (see (\ref{V_nalt}))
\begin{equation}
V_n'=\{f:(-1,1)\rightarrow \mathbb{C}\vert f \in AC_{\textrm{loc}}^{(n-1)}(-1,1); (1-x^2)^{n/2}f^{(n)} \in L^2(-1,1)\}.
\end{equation}
Then $V_n=V_n'$.
\end{theorem}
\begin{proof}
Clearly we need to only establish $V_n' \subset V_n$. It suffices to prove this theorem with the interval $(-1,1)$ replaced by $(0,1)$; a similar proof holds for the interval $(-1,0)$.
Let $f \in V_n'$.
\medskip
\\\underline{Case 1} $n=1$\\
We need to show 
\begin{equation*}
(1-x)^{1/2}f' \in L^2(0,1) \implies f \in L^2(0,1).
\end{equation*}
Since
\begin{equation*}
f(t)-f(0)= \int_0^t (1-u)^{-1/2}\left((1-u)^{1/2}f'(u)\right) du,
\end{equation*}
we apply Theorem \ref{Integral Inequality Theorem} with
\[
\varphi(t)=(1-t)^{-1/2} \quad\text{and}\quad \psi(t)=1 \quad(t \in (0,1)).
\]
Both $\varphi$ and $\psi$ satisfy the conditions of Theorem \ref{Integral Inequality Theorem}; furthermore a calculation shows
\begin{equation*}
\left(\int_0^t \varphi^2(u)du\right)\left(\int_t^1 \psi^2(u)du\right) =(t-1)\ln(1-t),
\end{equation*}
which is bounded on $(0,1)$.
By Theorem \ref{Integral Inequality Theorem}, $f \in L^2(0,1)$.
\medskip
\\\underline{Case 2} $n>1$ \\
We first show 
\begin{equation}
(1-t)^{n/2}f^{(n)} \in L^2(0,1) \implies (1-t)^{(n-1)/2}f^{(n-1)} \in L^2(0,1).
\label{Inductive step}
\end{equation}
Since
\begin{equation*}
f^{(n-1)}(t)-f^{(n-1)}(0)=\int_0^t (1-u)^{-n/2}\left((1-u)^{n/2}f^{(n)}(u)\right)du, 
\end{equation*}
we see that
\begin{align*}
&(1-t)^{{(n-1)/2}}f^{(n-1)}(t)-(1-t)^{{(n-1)/2}}f(0) \\
=&(1-t)^{(n-1)/2}\int_0^t (1-u)^{-n/2}\left( (1-u)^{n/2}f^{(n)}(u)\right) du.
\end{align*}
Clearly, $(1-t)^{(n-1)/2}f^{(n-1)}(0) \in L^2(0,1)$.
Let 
\[
\varphi(t)=(1-t)^{-n/2} \quad\text{and}\quad \psi(t)=(1-t)^{(n-1)/2} \quad (0<t<1);
\]
both functions satisfy the conditions of Theorem \ref{Integral Inequality Theorem}. Moreover, for $t \in (0,1)$,
\begin{align*}
&\left(\int_0^t \varphi^2(u)du\right)\left(\int_t^1 \psi^2(u)du\right)\\
&=\left(\int_0^t (1-u)^{-n}du\right)\left(\int_t^1(1-u)^{n-1}du\right)
=\dfrac{(1-t)(1-(1-t)^{n-1}}{n(n-1)},
\end{align*}
which is bounded on $(0,1)$.\\
It follows that
\begin{equation*}
(1-t)^{(n-1)/2}f^{(n-1)} \in L^2(0,1),
\end{equation*}
establishing (\ref{Inductive step}).\\
By repeating this argument, we see that for $0 \le j \le n$,
\begin{equation}
(1-t)^{n/2}f^{(n)} \in L^2(0,1) \implies (1-t)^{(n-j)/2}f^{(n-j)} \in L^2(0,1). \label{Implication}
\end{equation}
Replacing $j$ by $n-j$ in (\ref{Implication}) shows
\begin{equation*}
(1-t)^{n/2} \in L^2(0,1) \implies (1-t)^{j/2}f^{(j)} \in L^2(0,1) \quad \text{for} \quad 0 \le j \le n.
\end{equation*}
This completes the proof of the theorem.
\end{proof}

In particular, this theorem yields a new proof of the identity in (\ref{Domain(A^n)}) (or \ref{DomainA^n}); that is,  
\begin{align}
\begin{split}
V_{2n}&=\mathcal{D}(A^n)\\
&=\{f:(-1,1) \rightarrow \mathbb{C} \vert f \in AC_{\textrm{loc}}^{(2n-1)}(-1,1); (1-x^2)^{n}f^{(2n)} \in L^2(0,1)\}.
\end{split}
\label{Domain of A^n}
\end{align}

In the above proof, we note that we `de-incremented' powers of $(1-x^2)^{n/2}$ by integral multiples of $1/2$. 
In the next theorem, which establishes optimal smoothness conditions for $f \in \mathcal{D}(A^n)$, we de-increment powers of $(1-x^2)^{n}$ by integral multiples of $1$.

\begin{theorem}
\label{Main Theorem}For $n\in\mathbb{N},$ let $f\in\mathcal{D}(A^{n}),$ where
$A$ is the Legendre operator defined in (\ref{LegendreSA}) and where
$\mathcal{D}(A^{n})$ is given as in (\ref{Domain of A^n}). Then

\begin{enumerate}
\item[(i)]
\begin{equation}
f\in\mathcal{D}(A^{n})\Longrightarrow f^{(n)}\in L^{2}(-1,1);
\label{Legendre_main_result}%
\end{equation}

\item[(ii)] $f,f^{\prime},\ldots f^{(n-1)}\in AC[-1,1];$

\item[(iii)] the implication in (\ref{Legendre_main_result}) is best possible in the sense that there exists $g\in\mathcal{D}(A^{n})$ such that $g^{(n+1)}\notin
L^{2}(-1,1).$
\end{enumerate}

\begin{remark}
As mentioned earlier, the case $n=1$ was established in \cite{ELM-Legendre} and the case $n=2$ was shown in \cite{Littlejohn-Wicks} and \cite{Wicks}.
\end{remark}

\begin{proof}
Fix both $n\in\mathbb{N}$ and $f\in\mathcal{D}(A^{n}).$ Then, from
(\ref{Domain of A^n}),%
\[
(1-t^{2})^{n}f^{(2n)}\in L^{2}(-1,1).
\]
For part (i), it suffices to prove $f^{(n)}\in L^{2}(0,1);$ a similar argument
shows $f^{(n)}\in L^{2}(-1,0).$ Furthermore, observe
\[
(1-t^{2})^{n}f^{(2n)}\in L^{2}(0,1)\Longleftrightarrow(1-t)^{n}f^{(2n)}\in
L^{2}(0,1).
\]
Consequently, we first prove the implication
\begin{equation*}
(1-t)^{n}f^{(2n)}\in L^{2}(0,1)\Rightarrow(1-t)^{n-1}f^{(2n-1)}\in L^{2}(0,1).
\label{Legendre_main_step}%
\end{equation*}
It will follow, by iteration, that%
\begin{equation*}
(1-t)^{n-j}f^{(2n-j)}\in L^{2}(0,1)\quad(j=0,1,\ldots n),
\label{Legendre_main_implication}%
\end{equation*}
and, in particular with $j=n,$%
\begin{equation*}
f^{(n)}\in L^{2}(0,1). \label{Legendre_main_result_[0,1]}%
\end{equation*}
One integration shows%
\begin{align*}
(1-t)^{n-1}f^{(2n-1)}(t)  &  =(1-t)^{n-1}f^{(2n-1)}(0)\\
&  +(1-t)^{n-1}\int_{0}^{t}\frac{1}{(1-u)^{n}}\left(  (1-u)^{n}f^{(2n)}%
(u)\right)  du
\end{align*}
so it suffices to prove%
\begin{equation}
(1-t)^{n-1}\int_{0}^{t}\frac{1}{(1-u)^{n}}\left(  (1-u)^{n}f^{(2n)}(u)\right)
du \in L^2(0,1). \label{Legendre-key_step}%
\end{equation}
To this end, we apply Theorem \ref{Integral Inequality Theorem} with
$(1-t)^{n}f^{(2n)}\in L^{2}(0,1),\varphi(t)=(1-t)^{-n}$, $\psi(t)=(1-t)^{n-1}$
and $w\equiv1$ on $(a,b)=(0,1).$ Both $\varphi$ and $\psi$ satisfy the
conditions in Theorem \ref{Integral Inequality Theorem}. 
Moreover, observe that for $t\in(0,1),$
\begin{align*}
K^{2}(t)  &  =\left(  \int_{0}^{t}\varphi^{2}(u)du\right)  \left(  \int%
_{t}^{1}\psi^{2}(u)du\right) \\
&  =\left(  \int_{0}^{t}(1-u)^{-2n}du\right)  \left(  \int_{t}^{1}%
(1-u)^{2n-2}du\right) \\
&  =\left(  \frac{(1-t)^{-2n+1}}{2n-1}-\frac{1}{2n-1}\right)  \left(
\frac{(1-t)^{2n-1}}{2n-1}\right)
\end{align*}
which is clearly bounded on $(0,1)$. 
This proves (\ref{Legendre-key_step}). As
argued above, putting all of the above steps together, we see that
\[
(1-t^{2})^{n-j}f^{(2n-j)}\in L^{2}(-1,1)\quad(j=0,1,\ldots n)
\]
so, when $j=n,$ we obtain $f^{(n)}\in L^{2}(-1,1)$ as claimed. This proves
part (i) of the theorem.
\medskip\newline Part (ii) follows immediately from
(i) and standard results on absolute continuity in real analysis. 
\medskip\newline To prove part
(iii), define $g:[-1,1]\rightarrow\mathbb{C}$ by%
\begin{equation}
g(t)=\left\{
\begin{array}
[c]{ll}%
p(t) & \text{if }-1<t \leq 0\\
(1-t)^{n}\ln(1-t) & \text{if }0<t<1,
\end{array}
\right.
\label{function g}
\end{equation}
where $p(t)$ is any polynomial such that $g\in C^{n}(-1,1)$. 
Calculations show that, for $0 < t <1$, there exists constants $a_j=a_{j,n}$ such that 
\[
g^{(j)}(t)=(-1)^j\dfrac{n!}{(n-j)!}(1-t)^{n-j}\ln(1-t)+a_j(1-t)^{n-j} \quad(j=0,1,\cdots n)
\]
In particular,
\begin{equation*}
g^{(n)}(t) =(-1)^nn!\ln(1-t) +a_{n,n}\quad(t \in (0,1)) \label{nth derivative of g}.
\end{equation*}
so
\begin{equation*}
g^{(n+1)}(t)=\dfrac{(-1)^{n+1}n!}{1-t}\quad(t \in (0,1)).
\end{equation*}
This shows that $g^{(n+1}) \notin L^2(-1,1)$.
Continuing to differentiate we find, for $0<t<1$,
\[
g^{(n+j)}(t)=\dfrac{(-1)^{n+1}n!(j-1)!}{(1-t)^j} \quad(j=0,1,2,\cdots, n)
\]
and, in particular,
\[
g^{(2n)}(t)=\dfrac{(-1)^{n+1}n!(n-1)!}{(1-t)^n}\quad(t \in (0,1)).
\]
Hence, referring to the definition of $g$ in (\ref{function g}), we see that 
\[
(1-t)^ng^{(2n)} \in L^2(-1,1).
\]
Consequently, from (\ref{Domain of A^n}), $g\in \mathcal{D}(A^n)$. This completes the proof.
\end{proof}
\end{theorem}

\section{Concluding Remarks\label{Concluding Remarks}}
To summarize, we have two substantially different, but equivalent, characterizations of $\mathcal{D}(A^n)$, when $n \in \mathbb{N}$, of the $n^{th}$ composite power of the Legendre operator $A$ defined in (\ref{Operator A}). 

Indeed, the GKN theory yields
\begin{align}
\begin{split}
\mathcal{D}(A^n)=\{f:(-1,1) \rightarrow \mathbb{C} \vert &f,f',\cdots,f^{(2n-1)} \in AC_{\mathrm{loc}}(-1,1); f,\ell^n[f] \in L^2(-1,1); \\
&[f,g_1]_{2n}=[f,g_2]_{2n}=\cdots=[f,g_{2n}]=0\}
\label{GKN char}
\end{split}
\end{align}
where $[\cdot,\cdot]_{2n}$ is defined in (\ref{symplectic2n}) and (\ref{Legendre Symplectic}) and the functions $\{g_j\}_{j=1}^{2n}$ are given in Remark \ref{GKN sets for Legendre}.

On the other hand, an application of the left-definite theory shows
\begin{align}
\mathcal{D}(A^n)=\{f:(-1,1) \rightarrow \mathbb{C} \mid &f,f',\cdots f^{(2n-1)} \in AC_{\textrm{loc}}(-1,1); \label{LD char}\\
&(1-t^2)^nf^{(2n)} \in L^2(-1,1)\} \notag.
\end{align}
To check whether or not a certain function $f$ belongs to $\mathcal{D}(A^n)$, we remark that (\ref{LD char}) seems more efficient than applying the GKN approach leading to (\ref{GKN char}). We also established optimal and best-possible smoothness conditions for $f \in \mathcal{D}(A^n)$.

\end{document}